\documentclass[12pt]{amsart}

\usepackage{latexsym}
\usepackage{amssymb}
\usepackage[cp850]{inputenc}
\usepackage{epsfig}
\usepackage{psfrag}
\usepackage{amsthm}
\usepackage{amscd}
\usepackage{amsmath}
\usepackage{amsfonts}
\usepackage{graphics}
\usepackage[all]{xy}

\newtheorem{sat}{Theorem}[section]		\newtheorem{lem}[sat]{Lemma}
\newtheorem{kor}[sat]{Corollary}			\newtheorem{prop}[sat]{Proposition}
				\newtheorem{defi}{Definition}
\newtheorem*{defi*}{Definition}			\newtheorem*{bei*}{Example}
\newtheorem*{sat*}{Theorem}				\newtheorem*{kor*}{Corollary}
\newtheorem*{rmk*}{Remark}					
\newtheorem{fact}[sat]{Fact}	

\let\ssection=\section
\renewcommand{\section}{\setcounter{equation}{0}\ssection}

\newtheorem*{namedtheorem}{\theoremname}
\newcommand{\theoremname}{testing}
\newenvironment{named}[1]{\renewcommand{\theoremname}{#1}\begin{namedtheorem}}{\end{namedtheorem}}

\theoremstyle{remark}

\newcommand{\BC}{\mathbb C}			\newcommand{\BH}{\mathbb H}
\newcommand{\BR}{\mathbb R}			
			
\newcommand{\BS}{\mathbb S}			\newcommand{\BZ}{\mathbb Z}
\newcommand{\BF}{\mathbb F}

\newcommand{\CE}{\mathcal E}

\newcommand{\CO}{\mathcal O}

		\newcommand{\CV}{\mathcal V}

\newcommand{\D}{\partial}
\newcommand{\into}{\hookrightarrow}
\newcommand{\DD}{\nabla}

\DeclareMathOperator{\PSL}{PSL}		
\DeclareMathOperator{\Isom}{Isom}	
\DeclareMathOperator{\vol}{vol}		

\DeclareMathOperator{\inj}{inj}
\DeclareMathOperator{\diam}{diam}
\DeclareMathOperator{\rank}{rank}

\DeclareMathOperator{\length}{length}
\DeclareMathOperator{\diver}{div}
\newcommand{\comment}[1]{}
\DeclareMathOperator{\Fr}{Fr}
\DeclareMathOperator{\dist}{dist}
\DeclareMathOperator{\area}{area}
\begin{document}

\title[]{A finiteness theorem for hyperbolic 3-manifolds}
\author{Ian Biringer \& Juan Souto}
\thanks{The first author has been partially supported by NSF Postdoctoral Research Fellowship DMS-0902991. The second author has been partially supported by the NSF grant DMS-0706878 and the Alfred P. Sloan Foundation.}
\begin{abstract}
We prove that there are only finitely many closed hyperbolic 3-manifolds with injectivity radius and first eigenvalue of the Laplacian bounded below whose fundamental groups can be generated by a given number of elements.  Our techniques also have particular application to arithmetic manifolds.

\vspace{3mm}

\noindent \scriptsize{MATHEMATICS SUBJECT CLASSIFICATION:}  \small 57M50
\end{abstract}

\maketitle

\section{Introduction}

A common pursuit in differential geometry is to bound the number of closed $n $-manifolds that admit a Riemannian metric with controlled geometry: for instance, one might specify constraints on diameter, curvature, volume or injectivity radius \cite {Cheegerfiniteness}, \cite {GPW}.  If one only considers locally symmetric metrics, these finiteness theorems combine with Mostow's rigidity theorem to yield much stronger conclusions.   As an example, Wang's finiteness theorem \cite{Wang} asserts that for $n\ge 4$ and $V >0 $, there are finitely many isometry classes of hyperbolic $n$-manifolds $M$ with volume at most $V$.  Wang's theorem still holds when $n = 3 $ if in addition to the volume bound, one only considers manifolds $M$ with  injectivity radius $\inj(M)\ge\epsilon>0$ \cite [Theorem E.2.4]{Benedetti-Petronio}. 

Our goal is to provide a finiteness result for hyperbolic 3-manifolds with constrained injectivity radius, first eigenvalue of the Laplacian and rank of the fundamental group. Here, the \it rank \rm of a group is the minimal number of elements needed to generate it.  We prove:

\begin{sat}\label{thm:main}
For every $\epsilon,\delta,k>0$, there are only finitely many isometry classes of closed hyperbolic 3-manifolds $M$ with injectivity radius $\inj(M)\ge\epsilon$, first eigenvalue of the Laplacian $\lambda_1(M)\ge\delta$ and $\rank (\pi_1(M)) \leq k$.
\end{sat}

It is not hard to see that Theorem \ref {thm:main} fails if any of the three constraints are dropped.  First, any closed hyperbolic $3$-manifold given by a fibration $\Sigma_g \to M \to \BS^ 1 $ has an infinite sequence of cyclic covers with $\inj (M_i) \geq \inj (M) $ and $\rank(\pi_1 M_i) \leq 2g + 1 $.   Second, the congruence covers of any closed arithmetic hyperbolic $3$-manifold have injectivity radius bounded below and $\lambda_1 \geq \frac 34$, by Burger-Sarnak \cite {Burger-Sarnak}.  Finally, applying Thurston's Dehn filling theorem \cite {Benedetti-Petronio} to any noncompact finite volume hyperbolic $3$-manifold $N $ gives an infinite sequence $(M_i) $ of closed hyperbolic $3$-manifolds with $\rank(\pi_1 M_i) \leq \rank (\pi_1 N)$ and $\lambda_1 (M_i) \ge \delta > 0 $.  To obtain the bound on $\lambda_1 (M_i) $, one uses Thurston's theorem to arrange that $M_i \to N $ in the Gromov-Hausdorff topology;  since Gromov-Hausdorff convergence of hyperbolic $3$-manifolds is $C^\infty $ \cite [Remark E.1.19] {Benedetti-Petronio}, it follows that $\lambda_1 (M_i) \to \lambda_1 (N) > 0 $. 

The proof of Theorem \ref{thm:main} goes as follows. Assume that we have a sequence $(M_i)$ of pairwise-distinct, closed hyperbolic 3-manifolds with $\inj(M_i)\ge\epsilon$, and suppose that each $\pi_1(M_i)$ can be generated by $k$ elements. We will show in Theorem \ref{lem:degenerate} that after passing to a subsequence, there are base points $p_i\in M_i$ such that the sequence of pointed manifolds $(M_i,p_i)$ converges in the Gromov-Hausdorff topology to a pointed manifold $(M_\infty,p_\infty)$ which has a degenerate end. It follows from Proposition \ref{Laplacianconvergence} that $\lambda_1(M_i) \to 0$.  Specifically, one shows that the Cheeger constants $h(M_i) \to 0$ and then applies a result of Buser \cite{Buser} to say the same for $\lambda_1(M_i)$.  

\medskip

Surprisingly, although our techniques are very geometric they have particular application to arithmetic manifolds.  In the last section, we prove the following result and several corollaries.

\begin{sat}\label{arithmetic-finite}
For all $\epsilon, k >0$, there are only finitely many commensurability classes of closed arithmetic hyperbolic 3-manifolds $M$ with $\inj(M)\ge\epsilon$ and $\rank(\pi_1(M))\le k$.
\end{sat}

Our work is structured as follows.   After the introduction, we begin with a discussion of the eigenvalues of the Laplace-Beltrami operator and prove a weak version of Buser's inequality \cite{Buser} for hyperbolic $3$-orbifolds.   Section \ref{geosec} recalls well-known facts about Gromov-Hausdorff convergence and Section \ref{degenendssec} discusses degenerate ends and their relation to $\lambda_1$.  We introduce \it carrier graphs\rm, the main technical tool of this note, in Section \ref{carriersec} and in the next section use them to produce degenerate ends in Gromov-Hausdorff limits of manifolds with bounded rank and injectivity radius.  Finally, we show how our techniques are particularly effective when applied to arithmetic manifolds.  The arithmetic applications require a variant of Thurston's covering theorem, which is proved in an appendix.

\medskip

\noindent \bf Acknowledgements: \rm This work would have not been possible without the ideas of Ian Agol. The second author would like to thank the Universidad Aut\'onoma de Madrid for its hospitality while this paper was being written, and the first is indebted to Lisa Wang for her help typesetting.  We would especially like to thank the referee, without whom this paper would be half as long and far less than half as comprehensible.

\section{Eigenvalues of the Laplacian}\label{sec-proof}
Let $M$ be a closed hyperbolic 3-manifold. The {\em Laplacian} $\Delta f$ of a smooth function $f : M \to \BR$ is defined as 
$$\Delta f=-\diver\DD f$$
The Laplacian extends to a self-adjoint linear operator $\Delta_M $ on the Sobolev space $H^ 1(M)$.  It is well-known that the spectrum of this operator is a discrete subset of $[0,\infty) $; furthermore, $0 $ is an eigenvalue with ($1 $-dimensional) eigenspace the set of constant functions on $M $, and each other eigenspace is finite dimensional.  Let
$$0=\lambda_0(M)<\lambda_1(M)\le\lambda_2(M)\le\dots$$
be the eigenvalues of $\Delta_M$ in increasing order, listed so that repetitions indicate multiplicity.   

By work of Buser \cite {Buser} and Cheeger \cite {Cheeger}, the first nontrivial eigenvalue $\lambda_1 (M) $ is strongly tied to the \it Cheeger constant \rm of $M $; this is defined as
$$h(M)=\inf_{{\tiny
\begin{array}{c}
U \subset M
\end{array}
}}
\frac{\area(\partial U)}{\min\{\vol(U),\vol(M \setminus U)\}}$$ 
where the infimum is taken over smooth $3 $-dimensional submanifolds with boundary inside $M $.  Their work gives the following explicit relationship:
\begin{equation}\label{eq:buser}
\frac 14 h (M)^ 2 \le \lambda_1(M)\le 4h(M)^2+10h(M)
\end{equation}  Here, the first inequality is due to Cheeger and the second to Buser.  For us, the relevant implication of (\ref {eq:buser}) is that when $M $ is a closed hyperbolic $3$-manifold then $\lambda_1 (M) \approx 0   $ if and only if $  h (M) \approx 0 $.

Our work here requires only a very weak version of Buser's inequality.  Since this is easy to prove, we make our exposition self-contained by recording it below.  Another reason to do this is that we will need some version of Buser's result for \it orbifolds\rm, and this has not yet been written down.

Recall that a hyperbolic $3 $-orbifold is a metric quotient $\CO = \BH^ 3 / \Gamma, $  where $\Gamma < \Isom (\BH^3) $ is some discrete group of isometries of hyperbolic $3 $-space.  One can define a Laplacian operator $\Delta_{\CO} $ by letting $$H^ 1 (\CO) := \big \{  f : \CO \to \BR \ \big| \ f \text { lifts to } \tilde f \in H^ 1_{\text {loc} } (\BH^ 3) \ \big \} $$ and $\Delta_{\CO } $ be the operator on $H^ 1 (\CO) $ given by applying the Laplacian $\Delta_{\BH^ 3 } $  to the ($\Gamma $-invariant) lift $\tilde f$ and descending the $\Gamma $-invariant result to a map $ \CO \to \BR $. As before, $\Delta_\CO $ is self-adjoint with discrete, real spectrum
$$0=\lambda_0(\CO)<\lambda_1(\CO)\le\lambda_2(\CO)\le\dots$$ 
  If $\Gamma $ is torsion free, so that $\CO $ is a hyperbolic $3$-manifold, then this definition of $\Delta_\CO $ agrees with that given before.  

\begin {prop}
\label {easyBuser}
Assume that $\CO $ is a closed hyperbolic $3 $-orbifold and $U \subset \CO $ is an open set such that
\begin{enumerate}
\item the frontier $\Fr (U) $ of $U $ can be partitioned into $n $ subsets, each with diameter at most $D $,
\item both $U $ and $\CO \setminus U $ have volume at least $V $.
\end{enumerate}
Then if $\CV_{D + 1 } $ is the volume of a ball of radius $D+ 1 $ in $\BH^ 3 $, we have $$\lambda_1 (\CO) \leq \frac {n \CV_{D + 1 }} {V - n\CV_{D + 1 } }.$$ 
\end {prop}

Consequently, $\lambda_1 (\CO) $ is small when $\CO$ can be divided into two large volume pieces using a small number of sets that have small diameter.

\begin{proof}
The eigenspace of $0$ is the space of constant functions on $\BH^ 3 $.  To compute $\lambda_1 (M ) $, one can then take an infimum of \it Rayleigh quotients \rm of functions in $H^ 1 (\CO) $ orthogonal to the constant functions \cite [pg. 16]{chavel}: $$\lambda_1 (\CO) = \inf \Big \{ \frac {\int ||\nabla f ||^2} {  \int f^2 } \ \big | \ f \in H^1 (\CO) \text { with } \int f = 0  \ \Big\} . $$ Let $N $ be the $1$-neighborhood of $\Fr (U) $ in $\CO $.  Then we can define $f: \CO \to \BR$ by $$ f (x) = \begin{cases}  1 &  x \notin N, x \in U\\ \dist (x,\Fr (U)) & x \in N , x \in U \\ - \dist (x,\Fr (U)) & x \in N , x \notin U \\ -1 & x \notin N , x \notin U \end{cases} .$$  The function $f $ is $1 $-lipschitz, so it lies in $H^ 1 (\CO) $  and is differentiable almost everywhere.  Its gradient vanishes outside $N $, and  $| | \nabla f | | \leq 1 $ everywhere within $N $ that it is defined.  Therefore,
\begin{eqnarray*}
\int | | \nabla f | | ^ 2 & \leq & \vol (N) \\ & \leq &  n \CV_{D + 1 }.
\end {eqnarray*}
Furthermore, we have both $$\int_{\CO \setminus U } f^ 2 , \  \int_{U } f^ 2 \ \  \geq \ \ V - n \CV_{D + 1 }.$$  In particular, the Rayleigh quotient of $f $ (more than) satisfies the inequality given in the statement of the Proposition.

We are not quite finished, though, since $f $ might not integrate to $0 $.  So, assume without loss of generality that $\int_{\CO \setminus U } f^ 2 < \  \int_{U } f^ 2 $.  Then we can create a new function $f' : \CO \to \BR $ by letting
$$f' \big|_U = \Big (\frac{\int_{\CO \setminus U } f^ 2 } { \int_{U } f^ 2} \Big) \ f \big|_U \ \ \text { and } \ \  f'\big|_{\CO \setminus U }= f \big|_{\CO \setminus U } .$$  Then $f' $ will have zero integral, and one easily checks that its Rayleigh quotient is less than or equal to the upper bound desired for $\lambda_1 (\CO) $.
\end{proof}

\section { Geometric Convergence }
\label{geosec}
To understand the geometry of a particular family of closed hyperbolic $3$-manifolds, it is often useful to study (non-compact) hyperbolic $3$-manifolds that arise as  \it limits \rm of sequences in that family.  We recall here some tools from the theory of geometric limits that will find application later; unless otherwise stated, the material in this section can be found in \cite [Section E.1] {Benedetti-Petronio} or \cite [Chapter 7] {Japanese}.

Recall that hyperbolic $3$-manifolds are metric quotients of $\BH^ 3 $ by discrete, torsion-free groups of isometries; we first discuss convergence for subgroups of $\Isom(\BH^ 3) $. A sequence of closed subgroups $\Gamma_i \subset \Isom(\BH^ 3) $ converges in the \it Chabauty topology \rm to $\Gamma_\infty \subset \Isom(\BH^ 3) $ if
\begin {enumerate}
\item $\Gamma $ contains all accumulation points of sequences $(\gamma_i) $, $\gamma_i \in \Gamma_i $,
\item every $\gamma \in \Gamma_\infty $ is the limit of some sequence $(\gamma_i)$, with $\gamma_i \in \Gamma_i $.
\end{enumerate}
When $\Gamma_i \to \Gamma_\infty $, it is easy to see that $\Gamma_\infty $ is a closed subgroup of $\Isom(\BH^ 3) $.  In fact, the space of closed subgroups of $\Isom(\BH^ 3) $ is compact with respect to the Chabauty topology.  We are primarily interested in limits of discrete (and often torsion-free) subgroups, however.  While these do not form a closed subspace, the following well-known fact constrains how sequences of such groups can degenerate.

\begin{fact}[see Prop 7.2, \cite{Japanese}] Assume that $\Gamma_i \subset \Isom (\BH^ 3) $ is a sequence of discrete (and torsion-free) subgroups that converges in the Chabauty topology to some closed subgroup $\Gamma_\infty \subset \Isom (\BH^ 3) $.  Then either
\begin{enumerate}
\item there exists a sequence $\gamma_i \in \Gamma_i $ with $\gamma_i \to id $ and $\Gamma_\infty $ is virtually abelian, or
\item there is no such sequence and $\Gamma_\infty $ is discrete (and torsion-free).
\end{enumerate}
\label {degeneracy}
\end{fact}

The Chabauty topology on the space of discrete subgroups $\Gamma \subset \Isom (\BH^ 3) $ is related to the Gromov-Hausdorff topology on the space of quotient orbifolds $\CO_\Gamma = \BH^ 3 / \Gamma $.  To see this, fix a point $p \in \BH^3$.  Then each $\CO_{\Gamma_i }$ is naturally a \it pointed hyperbolic $3$-orbifold: \rm the projection of $p$ gives a preferred basepoint $p_\Gamma \in \CO_\Gamma$. 

\begin{defi}\rm A sequence of pointed hyperbolic $3 $-orbifolds $(\CO_i,p_i) $ converges in the \emph{pointed Gromov-Hausdorff topology} to $(\CO_\infty,p_\infty) $ if for every compact $K \subset \CO_\infty $ containing $p_\infty $, there is a sequence of $\lambda_i$-bilipschitz maps
\begin{equation}\label{eq:almost-isometric}
\phi_i:(K,p_\infty)\to(\CO_i,p_i)
\end{equation}
 with $1 \leq \lambda_i \leq \infty $ and $\lambda_i \to 1 $ as $i \to \infty $.  We will call $(\phi_i) $ a sequence of \it almost isometric maps \rm coming from Gromov-Hausdorff convergence.
\end{defi}
We then have:

\begin{sat}
If a sequence of discrete subgroups $\Gamma_i \subset \Isom(\BH^3)$ converges to a discrete subgroup $\Gamma \subset \Isom(\BH^3)$, then the pointed $3$-orbifolds $(\CO_{\Gamma_i},p_{\Gamma_i})$ converge in the Gromov-Hausdorff topology to $(\CO_{\Gamma},p_{\Gamma})$.  

Conversely, if a sequence of pointed hyperbolic $3$-orbifolds $(\CO_i, p_i)$ converges to $(\CO_\infty, p_\infty)$, there are discrete subgroups $\Gamma_i < \Isom \BH^3$ with $(\CO_i, p_i) \cong (\CO_{\Gamma_i}, p_{\Gamma_i})$ and $\Gamma_i \to \Gamma_\infty$ in the Chabauty topology.  


\label{GHvsCH}
\end{sat}

The following useful result comes from translating Fact \ref{degeneracy} (2) into the language of orbifolds and Gromov-Hausdorff convergence. 

\begin{kor}
Fix $\epsilon > 0$.  The set of closed, pointed hyperbolic $3$-manifolds $(M,p)$ such that $\inj(M,p) \geq \epsilon $ is compact in the pointed Gromov-Hausdorff topology. 
\label{thickcompact}
\end{kor}


\section { Degenerate ends and the Laplacian}
\label{degenendssec}
The next three sections require some basic facts from the theory of hyperbolic $3$-manifolds, namely the definition of the convex core and the geometric classification of ends.  

\begin{defi}
\rm The \it convex core \rm of a hyperbolic $3$-manifold $M$ is the smallest convex submanifold $CC ( M) $ whose inclusion into $M$ is a homotopy equivalence.  
Equivalently, if $M = \BH^3 / \Gamma $ then $ CC (M) $ is the projection to $M $ of the convex hull of the limit set $\Lambda(\Gamma) \subset \partial_\infty \BS^2_\infty $.  
\end{defi}

Note that the limit set associated to a closed hyperbolic manifold is all of $\BS^2_\infty $, so any such manifold is its own convex core.

The ends of noncompact hyperbolic $3$-manifolds fall into two geometric categories, depending on their relationships with the convex core. Specifically, let $M $ be a complete, infinite volume hyperbolic $3$-manifold with finitely generated fundamental group and no cusps.   An end $\CE$ of $M $ is called \it convex cocompact \rm if it has a neighborhood whose intersection with the convex core of $M $ is bounded, and {\em degenerate} otherwise.

The geometry of each of these types of ends is well understood.  First, the Tameness Theorem of Agol \cite{Agol} and Calegari-Gabai \cite {Calegari-Gabai} implies that every end $\CE $ of $M $ has a neighborhood $E $ that is homeomorphic to $\Sigma \times (0,\infty) $ for some closed surface $\Sigma $.   It is well-known that if $\CE $  is convex cocompact, then $E $ is bilipschitz to a warped product on $\Sigma \times (0,\infty)$ in which the metric on $\Sigma \times \{t\} $ is scaled by a factor exponential in $t $.  

On the other hand, we have the following well-known consequence of Canary's Filling Theorem \cite{Canary-covering}, Bonahon's Bounded Diameter Lemma \cite{Bonahon} and work of Freedman-Hass-Scott \cite [Theorem 2.5] {Canary-Minsky}.



\begin {fact}
\label{exitingsurfaces}
Every degenerate end of $M $ has a neighborhood $E$ homeomorphic to $\Sigma \times (0,\infty) $ in which every point lies within unit distance of a level surface with area at most $2\pi \chi (\Sigma) $.  
Here, a \emph{level surface} is any embedded surface homotopic to a fiber $\Sigma \times \{t\}$.

If $\inj (M) \geq \epsilon $, then $E $ can be chosen so that each of its points lies within unit distance of a level surface with diameter bounded above by some constant $C$ depending only on $\epsilon $ and $\chi (\Sigma) $.

In both cases, there is a sequence $(S_i)$ of such surfaces that exits $\CE$, meaning that every neighborhood of $\CE$ contains $S_i$ for large enough $i$.

\end{fact}

It follows that inside a degenerate end there are submanifolds with arbitrarily large volume that are bounded by surfaces with small  area.  Therefore, if one can find large pieces of a degenerate end inside of a \it closed \rm hyperbolic $3$-manifold, the Cheeger constant of that manifold will be small.  Buser's inequality (\ref {eq:buser}) will imply the same about the first eigenvalue of its Laplacian.

One way to formalize this is with the following Proposition.

\begin{prop}\label{prop:lambda-to-0}
Assume that $(M_i, x_i) $ is a sequence of pointed, closed hyperbolic $3$-manifolds that converges in the pointed Gromov-Hausdorff topology to a pointed hyperbolic $3$-manifold  $(M_\infty, x_\infty) $ that has a degenerate end.  Then $\lambda_1 (M_i) \to 0 $ as $i \to \infty $.
\label {Laplacianconvergence}
\end{prop}


\begin{proof}
The manifold $M_\infty$ has a degenerate end, so for each $\eta>0$ there is a compact submanifold $U \subset M_\infty $ with $\frac{\area(\D U)}{\vol(U)}<\eta$. Choose a compact subset $K \subset M_\infty $ containing both $x_\infty $ and $U $, and let $\phi_i : K \to M_i$ be a sequence of bilipschitz maps as in \eqref{eq:almost-isometric}. We have then
$$\lim_{i\to\infty}\frac{\area(\D(\phi_i(U)))}{\vol(\phi_i(U))}=\frac{\area(\D U)}{\vol(U)}<\eta. $$ Taking into account that the volume of $M_i \setminus \phi_i(U) $ grows without bound, we deduce that 
$$\limsup_{i\to\infty} h(M_i)\le\eta$$
Buser's inequality \eqref{eq:buser} implies that
$$\limsup_{i\to\infty}\lambda_1(M_i)\le 4\eta^2+10\eta$$
Since $\eta>0$ was arbitrary we conclude that $\lim_{ i \to \infty }\lambda_1(M_i)=0$.\end{proof}

All our applications of this result will be to sequences satisfying $\inj (M_i) \geq \epsilon $.  In those cases, one may appeal to Proposition \ref {easyBuser} instead of Buser's inequality to extend the conclusions above to orbifolds, as in the following technical proposition.  

\begin{prop}
Let $(\CO_i) $ be a sequence of pairwise-distinct compact hyperbolic $3 $-orbifolds that are covered by closed hyperbolic $3$-manifolds $(M_i) $. Assume that there are base points $x_i \in M_i $ such that $(M_i, x_i) $ converges in the pointed Gromov-Hausdorff topology to a pointed hyperbolic $3$-manifold $(M_\infty, x_\infty) $ that has a degenerate end and has $\inj (M_\infty) >0 $.  
Then $\lambda_1 (\CO_i) \to 0 $ as $ i \to \infty. $
\label{orbifoldlaplace}
\end{prop}


\begin{proof}
By Theorem \ref{GHvsCH}, we can assume that $M_i\cong \BH^3 / \Gamma_i$ and $\Gamma_i \rightarrow \Gamma_\infty$ in the Chabauty topology. The orbifolds $\CO_i$ are then isomorphic  to $\BH^3 / O_i$ for some discrete subgroups $O_i \subset \Isom(\BH^3)$ with $\Gamma_i \subset O_i$. Passing to a subsequence, we can assume that $(O_i)$ converges in the Chabauty topology to some closed subgroup $O_\infty \supset \Gamma_\infty$. Since $\Gamma_\infty$ is not virtually abelian, neither is $O_\infty$; so Fact \ref{degeneracy} implies that $O_\infty$ is discrete. Therefore, $(\CO_i)$ converges in the Gromov-Hausdorff topology to an orbifold $\CO_\infty= \BH^3 /O_\infty$ covered by $M_\infty$.   

Since the orbifolds $(\CO_i)$ are distinct, $\CO_\infty$ cannot be compact; the Thurston-Canary Covering Theorem \cite{Canary-covering} then states that the given degenerate end of $M_\infty$ has a neighborhood $E$ on which the orbifold covering map $\pi: M_\infty \to \CO_\infty$ is finite to one. Since $\inj(M_\infty) > 0$, Fact \ref{exitingsurfaces} gives a sequence of embedded level surfaces exiting $E$  that have uniformly bounded diameters. In particular, there are two embedded surfaces $S_1, S_2 \subset E$ that have diameter less than some fixed $K > 0$ and that bound a submanifold $U \subset M_\infty$ with volume bigger than any given $V >0$. 

The projection $\pi(U) \subset \CO_\infty$ has volume at least $\frac{V}{m}$, where $m $ is an upper bound for the degree of $\pi |_E$. The frontier of $\pi(U)$ is contained in the union of $\pi(S_1)$ and $ \pi(S_2)$; it therefore splits into two subsets, each with diameter at most $K $.

Let $\phi_i : \pi(U) \to \CO_i$ be a sequence of almost isometric maps coming from the Gromov-Hausdorff convergence $\CO_i \to \CO_\infty$. For large $i$, the volume of $\phi_i \circ \pi(U)$ is at least, say, $\frac{V}{2m}$ and its frontier partitions into two components of diameter less than $2K$. Also, $\vol(\CO_i) \to \infty$ so for large $i$ $$\vol(\CO_i \setminus \phi_i \circ \pi(U)) > \frac{V}{2M}.$$ Proposition \ref{easyBuser} then gives for large $i$: $$ \lambda_1(\CO_i) \leq \frac{2\CV_{2K}}{\frac{V}{2m} -2\CV_{2K}},$$ where $\CV_{2K}$ is the volume of a ball in $\BH^3$ of radius $2K$.  However, $V$ can be chosen arbitrarily large at the expense of increasing $i$, so $\lambda_1(\CO_i) \to 0$ as $i \to \infty$. 
\end{proof}

\section{Short graphs in manifolds with bounded rank}
\label{carriersec}
This section concerns \it carrier graphs\rm: technical tools that facilitates a geometric understanding of $\rank(\pi_1 M)$.   We first define them and record a few key properties, and then use them to study sequences of hyperbolic $3$-manifolds with bounded rank and injectivity radius.  Carrier graphs were first introduced by White in \cite{White}; variations of the techniques used here have been earlier exploited in \cite{Ian-rank,Juan-rank}. 

Let $M$ be a closed hyperbolic 3-manifold. A {\em carrier graph} consists of a metric graph $X$ and a 1-Lipschitz map
\begin{equation}\label{eq:carrier}
f:X\to M
\end{equation}
such that the induced homomorphism $f_*:\pi_1(X)\to\pi_1(M)$ is surjective.  The {\em length} of a subgraph $Y\subset X$ is defined to be the sum of the lengths of the edges it contains. 

Using the Arzela-Ascoli theorem, it is not hard to see that the set of carrier graphs with bounded total length in a given closed hyperbolic $3$-manifold is compact. In particular, there is one in each such $3$-manifold which has minimal length. In \cite{White}, White observed that these graphs have controlled geometry; for instance, they are trivalent with geodesic edges.  He used this to prove that if $f : X \to M $ is a {\em minimal length carrier graph} then $X$ has a circuit whose edge-length sum is bounded by some function of $\rank(\pi_1 X)$. 

In \cite{Ian-rank}, the first author extended White's result as follows:

\begin{prop}[Chains of bounded length]\label{prop:chains}
Let $M$ be a closed hyperbolic 3-manifold and $f:X\to M$ a minimal length carrier graph. Then we have a sequence of possibly disconnected subgraphs
$$\emptyset=Y_0\subset Y_1\subset\dots\subset Y_n=X$$
such that the length of any edge in $Y_{i+1}\setminus Y_i$ is bounded from above by some constant depending only on $\inj(M)$, $\rank(\pi_1X)$, $\length(Y_i)$ and the diameters of the convex cores of the covers of $M$ corresponding to $f_*(\pi_1(Y_i^j))$, where $Y_i^1,\dots,Y_i^{n_i}$ are the connected components of $Y_i$. Moreover, the number $n $ of subgraphs in the chain is bounded above by $3 (\rank (\pi_1 X)-1) $.
\end{prop}

Suppose now that a closed hyperbolic $3$-manifold $M$ has $\rank \pi_1 M \leq k$ and $\inj(M) \geq \epsilon$. Then Proposition \ref{prop:chains} shows that there is an upper bound $C(\epsilon, k)$ for the length of a minimal length carrier graph $X \to M$, unless one of the subgraphs $Y_i$ is associated to a large convex core. This simple idea, plus  some bookkeeping, gives the following useful lemma.

\begin{lem}\label{prop:graphs}
Assume that $(M_i)$ is a sequence of pairwise distinct closed hyperbolic $3$-manifolds with $\inj (M_i) \geq \epsilon $ and $\rank (\pi_1 (M_i)) \leq k $. Then there are a constant $L$ and a sequence $(Y_i)$ of metric graphs with 1-Lipschitz maps $(f_i:Y_i\to M_i)$ such that
\begin{enumerate}
\item $\rank(\pi_1(Y_i))\le k$,
\item $\length(Y_i)\le L$, and
\item $\lim_{i\to\infty}\diam(CC(\BH^3/(f_i)_*(\pi_1Y_i)))=\infty$.
\end{enumerate}
Here $CC(\BH^3/(f_i)_*(\pi_1 Y_i))$ is the convex core of the cover of $M_i$ corresponding to the image of the homomorphism $(f_i)_*:\pi_1Y_i \to\pi_1M_i$.
\end{lem}

\begin{proof}[Proof of Lemma \ref{prop:graphs}]
To begin with fix $\epsilon$, $k$ and a sequence of hyperbolic 3-manifolds $(M_i)$ as in the statement of the lemma.  Choose for each $i $ a minimal length carrier graph
$$f_i:X_i\to M_i$$
with $\rank(\pi_1(X_i))=\rank(\pi_1(M_i))=k$. 

Assume for the moment that the sequence $(\length(X_i))$ is bounded from above by some positive number $L$. In other words, the graphs $X_i$ themselves satisfy (1) and (2). On the other hand, we have by definition that $(f_i)_*(\pi_1(X_i))=\pi_1(M_i)$ and hence
$$CC(\BH^3/(f_{i})_*(\pi_1(X_{i})))=M_i$$
Since the sequence $(M_i)$ consists of pairwise distinct manifolds with $\inj(M_i)\ge\epsilon$ we obtain, for example, from Wang's finiteness theorem that $\vol(M_i)\to\infty$. The injectivity radius bound then implies that $\diam(M_i) \to \infty $ as well.  This means that the carrier graphs $f_i:X_i\to M_i$ themselves satisfy also (3). This concludes the proof if the sequence $(\length(X_i))$ is bounded. 

We treat now the general case. In the light of the above, we may assume without loss of generality that $\length(X_i)\to\infty$. Consider for each $i$ the chain 
\begin{equation}\label{eq:chain}
\emptyset=Y_0^i\subset Y_1^i\subset\dots\subset Y_{n_i}^i=X_i
\end{equation}
provided by Proposition \ref{prop:chains}. Since $\length(Y_0^i)=0$, $\length(X_i)\to\infty$ and the length $n_i $ of each chain is bounded independently of $i $, we can choose a sequence $(m_i)$ with
\begin{itemize}
\item[(a)] $0\le m_i\le n_i-1$,
\item[(b)] $\limsup_{i\to\infty}\length(Y^i_{m_i})<\infty$, and
\item[(c)] $\lim_{i\to\infty}\length(Y^i_{m_i+1})=\infty$.
\end{itemize}
Observe that by condition (b), any of the connected components $Z_1^i,\dots,Z_{r_i}^i$ of $Y^i_{m_i}$ satisfies (1) and (2) for any $L<\infty$ with  
$$\limsup_{i\to\infty}\length(Y^i_{m_i})<L$$
By Proposition \ref{prop:chains}, $\length(Y^i_{m_i+1})$ is bounded in terms of $k$, $L$ and 
$$\max_{j=1,\dots,r_i}\{\diam(CC(\BH^3/(f_i)_*(\pi_1(Z^i_j))))\}$$
Since $\length(Y^i_{m_i+1})$ tends to $\infty$ by condition (c), we obtain that there is a sequence of component of $Y_{m_i}^i$, say $Z_1^i$, with
$$\lim_i\diam(CC(\BH^3/(f_i)_*(\pi_1(Z^i_1))))=\infty$$
In other words, the sequence of maps $f_i\vert_{Z_1^i}:Z^i_1\to M_i$ satisfies (3). This concludes the proof of Lemma \ref{prop:graphs}
\end{proof}

\section{Limits of Thick Manifolds with Bounded Rank}

The main result of this section is that any sequence of pairwise distinct hyperbolic 3-manifolds with injectivity radius bounded from below and whose fundamental group has rank bounded from above has a Gromov-Hausdorff limit with a degenerate end.  Even better,
 
\begin{prop}
\label{lem:degenerate}
Assume that $(M_i)$ is a sequence of pairwise distinct hyperbolic 3-manifolds with $\inj(M_i)\ge\epsilon$ and $\rank(\pi_1(M_i))\le k$. Then there are points $x_i\in M_i$ such that, up to passing to a subsequence, the pointed manifolds $(M_i,x_i)$ converge in the pointed Gromov-Hausdorff topology to a pointed hyperbolic $3$-manifold $(M_\infty,x_\infty)$ homeomorphic to $\Sigma \times \BR $ that has two degenerate ends.  Here, $\Sigma $ is a closed, orientable surface with genus at most $k $.
\end{prop}

Before beginning its proof, we deduce from it Theorem \ref{thm:main}:

\begin{named}{Theorem \ref{thm:main}}
For every $\epsilon,\delta,k>0$, there are only finitely many isometry types of closed hyperbolic 3-manifolds $M$ with $\rank (\pi_1(M)) \leq k$, $\inj(M)\ge\epsilon$ and $\lambda_1(M)\ge\delta$.
\end{named}
\begin{proof}
Seeking a contradiction, suppose that there is a sequence of pairwise distinct hyperbolic 3-manifolds $(M_i)$ satisfying the assumptions of the theorem. By Proposition \ref{lem:degenerate} we can find points $x_i\in M_i$ such that, up to passing to a subsequence, the pointed manifolds $(M_i,x_i)$ converge in the pointed Gromov-Hausdorff topology to a manifold $(M_\infty,x_\infty)$ with a degenerate end. By Proposition \ref{prop:lambda-to-0}, we have $\lambda_1(M_i)\to 0$. This is a contradiction.
\end{proof}
 
It remains to prove Proposition \ref{lem:degenerate}. 

\begin{proof}[Proof of Proposition \ref{lem:degenerate}]
For each $i$, let $f_i:Y_i\to M_i$ be a sequence of graphs as provided by Lemma \ref{prop:graphs}.  Choose base points $y_i\in Y_i$ and set $x_i=f_i(y_i)\in M_i$.  Then since each $\inj(M_i) \geq \epsilon$, by Corollary \ref{thickcompact} we can pass to a subsequence so that $(M_i,x_i)$ converges in the Gromov-Hausdorff topology to some pointed hyperbolic 3-manifold $(M_\infty,x_\infty)$.

The assumption that the manifolds $M_i$ are pairwise distinct implies that $M_\infty$ is not compact. In particular, in order to show that it has a degenerate end, it suffices by Canary's extension of Thurston's covering theorem \cite{Canary-covering} to find a manifold $\tilde M_\infty$ which has a degenerate end and covers $M_\infty$. This is our goal. 

Passing to a subsequence, we may assume that each $\pi_1(Y_i,y_i)$ is isomorphic to the free group $\BF_m$ for some $m\le k$. There are then homomorphisms 
$$ \xymatrix{ \BF_m \ar[r]^-\cong & \pi_1(Y_i, y_i) \ar [r]^{(f_i)_*}  & \pi_1(M_i, x_i) \ar[r] & \Isom(\BH^3),}$$ but the first and last arrows are not canonically defined. However, since each $\length(Y_i) \leq L$, we may choose the first identification so that each element of the standard basis for $\BF_m$ is represented by a loop based at $y_i$ of length at most $2L$. For the last map, by Theorem  \ref{GHvsCH} there is a convergent sequence of groups $\Gamma_i \to \Gamma_\infty$ with $M_i \cong \BH^3 / \Gamma_i$ such that a fixed basepoint $x_{\BH^3} \in \BH^3$ projects to each $x_i$, for $i=1, \ldots , \infty$. We choose the last map $\phi: \pi_i(M_i, x_i) \to \Gamma_i$ so that a loop representing $\gamma \in \pi_1(M_i, x_i)$ lifts to a path in $\BH^3$ joining $x_{\BH^3}$ to $[\phi(\gamma)](x_{\BH^3})$. 

Composing yields a representation $$\rho_i:\BF_m\to\Isom(\BH^3), \ \ \ \rho_i(\BF_m)=\Gamma_i  \text{ and }   \BH^3 / \Gamma_i \cong M_i$$ such that if $e_j\in\BF_m$ is an element of the standard basis then for all $i$ we have
$$d_{\BH^3}((\rho_i(e_j))(x_{\BH^3}),x_{\BH^3})\le 2L.$$
	This implies that, up to passing to a further subsequence, the sequence of representations $(\rho_i)$ converges pointwise to a representation $\rho_\infty :\BF_m \to \Isom(\BH^3)$.  Since $\rho_i(\BF_m) = \Gamma_i $ and $\Gamma_i \to \Gamma_\infty$ in the Chabauty topology, it follows immediately from the definition given in Section \ref{geosec} that $\rho_\infty(\BF_m) \subset \Gamma_\infty$. In particular, $\rho_\infty(\BF_m)$ is discrete and the manifold $\tilde M_\infty=\BH^3/\rho_\infty(\BF_m)$ covers $M_\infty$. Observe that $\inj(\tilde M_\infty)\ge\epsilon$ and hence $\rho_\infty(\BF^m)$ does not contain parabolic elements.

We claim that $\tilde M_\infty$ has a degenerate end.  To begin with, its fundamental group is isomorphic to $\rho_\infty (\BF^m)$ and is therefore finitely generated.  The Tameness Theorem of Agol \cite{Agol-tameness} and Calegari-Gabai \cite{Calegari-Gabai} then implies that $\tilde M_\infty $ is homeomorphic to the interior of a compact $3$-manifold.  It then follows from Canary \cite{Canary-tameness} that either $\tilde M_\infty $ has a degenerate end or its convex core $CC (\tilde M_\infty) $ is compact.  Assuming the latter, Marden's stability theorem \cite{Marden74} implies then that there are bi-Lipschitz maps (defined for large enough $i$)
$$\tilde\phi_i:\BH^3/\rho_\infty(\BF_m)\to\BH^3/\rho_i(\BF_m)$$
whose bi-Lipschitz constants tends to $1$. This implies that 
$$\lim_{i\to\infty}\diam(CC(\BH^3/\rho_i(\BF_m)))=\diam(CC(\BH^3/\rho_\infty(\BF_m)))<\infty$$
contradicting that by Lemma \ref{prop:graphs} we have
$$\lim_{i\to\infty}\diam(CC(\BH^3/\rho_i(\BF_m)))=\lim_{i\to\infty}\diam(CC(\BH^3/(f_i)_*(\pi_1(Y_i))))=\infty$$

We have shown that the algebraic limit $\tilde M_\infty$ has a degenerate end $\CE $. One can bound the topology of $\CE $ in terms of $k $ as follows.  First, note that $\CE $ has a neighborhood homeomorphic to $\Sigma \times (0,\infty) $, where $\Sigma $ is a boundary component of some compact $3$-manifold $M $ with interior $\tilde M_\infty $.  The genus of a component of $\partial M $ is at most half the abelian rank (number of $\BZ $-summands) of $H_1 (\partial M) $.  The 'half lives, half dies' philosophy (e.g. \cite [Lemma 3.5]{Hatchernotes}) shows then that $g (\Sigma) $ is at most the abelian rank of $H_1 (M) $, which is at most $\rank (\pi_1 M) \leq k $.  

Canary's generalization of Thurston's covering theorem \cite {Canary-covering} implies that some neighborhood of $\CE $ finitely covers a neighborhood of a degenerate end $\CE' $ of the geometric limit $M_\infty $.  Furthermore, this end $\CE' $ has a neighborhood homeomorphic to $\Sigma' \times (0,\infty) $ for some surface $\Sigma' $ covered by $\Sigma $.  In particular, $g (\Sigma')\leq k $.  

We have now all but proved Proposition \ref {lem:degenerate}: our Gromov-Hausdorff limit $M_\infty $ has a degenerate end with topology bounded by $k $, but it may not be homeomorphic to a product.  However, combining Lemma \ref {lem:degenerate2} below with a diagonal argument remedies the situation and finishes the proof of Proposition \ref {lem:degenerate}.
\end{proof}

The following lemma is certainly well-known to experts, but it is not necessarily well-written.  So, we include a full proof here.

\begin{lem}\label{lem:degenerate2}
Assume that $M$ is a hyperbolic $3$-manifold with $\inj(M) > 0$ that has a degenerate end $\CE$ with a neighborhood $E \cong \Sigma \times \BR$.  Then if a sequence of points $(p_i)$ exits a degenerate end of $M$, any pointed Gromov-Hausdorff limit of any subsequence of $(M, p_i)$ is homeomorphic to $\Sigma \times \BR$ and has two degenerate ends.
\end{lem}

\begin{proof}
Assume that $(M, p_i)$ converges to some pointed hyperbolic $3$-manifold $(M_\infty, p_\infty)$. We first show that $M_\infty \cong \Sigma \times \BR$ by constructing a nested sequence of submanifolds $U_1\subset U_2 \subset \ldots$ with $\bigcup_k U_k = M_\infty$,  each $U_k \cong \Sigma \times [0, 1]$ and each inclusion $U_k \hookrightarrow U_{k+1}$ a homotopy equivalence. By Waldhausen's Cobordism Theorem \cite[Lemma 5.1]{Waldhausen}, each component of $U_{k+1} \setminus int(U_k)$ will be homeomorphic to $\Sigma \times [0,1]$, so a gluing argument will then show $M_\infty \cong \Sigma \times \BR$. 

Assume that $E \cong \Sigma \times \BR$ is a neighborhood of $\CE$ as given by Fact \ref{exitingsurfaces}, so that within unit distance from each point of $E$ there is a level surface with diameter less than some constant $C(\epsilon)$. Choose $D > C(\epsilon)+ 1$. We will construct a submanifold $U_1$ as above with $$B_D(p_\infty) \subset U_1 \subset B_{8D} (p_\infty). $$
To do this, choose $i$ large enough so there is a $2$-bilipshitz embedding $$ \phi : B_{8D} (p_\infty) \to M$$
 sending $p_\infty$ to $p_i$. Increasing $i$ as necessary, we may also assume that $p_i \in E$ and $\dist(p_i, \partial E) > 4D$.
 
 By Fact \ref{exitingsurfaces}, there is an embedded surface $S \subset E$ with $\dist (p_i, S) \leq 1$ and $\diam(S) \leq C(\epsilon).$ Then $S$ separates $M$, and there are points $a_1, a_2$ on opposite sides of $S$ with $$ \dist(a_1, p_i) = \dist( a_2, p_i)= 3D.$$  Then both $a_1, a_2 \in E$, so there are level surfaces $X_1$ and $X_2$ in $E$ with $\dist(a_k, X_k) \leq 1$ and diameters less than $C(\epsilon)$. 
 
 The surfaces $X_1, X_2$ and $S$ are all disjoint. By Waldhausen's Cobordism Theorem \cite[Lemma 5.1]{Waldhausen}, $X_1$ and $X_2$ bound a submanifold $V \subset M$ homeomorphic to $\Sigma \times [0,1]$. Since $S$ separates $X_1$ and $X_2$, we must have $S \subset V$. Moreover, this implies that $B_{2D}(p_i)$ (which contains $S$) intersects the interior of $V$, but not $\partial V$. We conclude that $B_{2D} (p_i) \subset V \subset B_{4D} (p_i)$; setting $U_1 = \phi^{-1} (V)$ gives a submanifold $U_1 \cong \Sigma \times [0,1]$ with $$B_D (p_\infty) \subset U_1 \subset B_{8D} (p_\infty).$$
Multiplying $D$ repeatedly by 8 and performing the same argument each time gives a sequence $$U_1 \subset U_2 \subset U_3 \subset \ldots$$ with $U_k \cong \Sigma \times [0,1]$ and $ \bigcup_k U_k = M_\infty$ as desired. 

We only have to check that the inclusions $U_k \into U_{k+1}$ are homotopy equivalences. However, in each submanifold $U_k$ there is an embedded surface $S_k$ with
\begin{itemize}
\item $\dist(p_\infty, S_k) \leq 2$
\item$ \diam(S_k) \leq 2C(\epsilon)$
\item $S_k \into U_k$ a homotopy equivalence.
\end{itemize}
In the notation we used earlier, $S_k$ is $\phi^{-1}(S)$. The diameter condition and proximity to $p_\infty$ imply that each $S_k$ is contained in every $U_1, U_2, \ldots $. In particular, $S_{k+1}$ is an incompressible embedded surface in $U_k \cong \Sigma \times [0,1]$, so the inclusion $S_{k+1} \into U_k$ is a homotopy equivalence. There, the same is true for $U_k \into U_{k+1}$. 

We know now that $M_\infty \cong \Sigma \times \BR$. To see that both ends are degenerate, one can simply observe from Gromov-Hausdorff convergence and Fact \ref{exitingsurfaces} that there is some $K > 0$ ($K \geq 3C(\epsilon)$ will do) such that through every point $x\in M_\infty$ there is an essential loop with length less than $K$. This implies that both ends are degenerate \cite{Japanese}.
 \end{proof}

\section{Corollaries for Arithmetic Manifolds}
\label {arithmetic}
In this final section, we prove some additional results concerning arithmetic hyperbolic $3 $-manifolds.  Although one can consult \cite{Reid-book} for a beautiful and detailed theory of such manifolds, we will need here only the following fact:   

\begin{sat}\label{vigneras}
Every closed arithmetic hyperbolic $3$-manifold $M$ covers some hyperbolic orbifold $\CO$ with $\lambda_1 (\CO) \geq \frac34$.
\end{sat}

This is a corollary of a deep result of Burger-Sarnak \cite[Corollary 1.3]{Burger-Sarnak}, generalizing work of Vigneras \cite{Vigneras} in dimension $2$. The statement given follows from theirs using a lemma of Long-Maclachlan-Reid \cite[Lemma 4.2]{long}.  A more detailed version of the history of this result is given by Agol in \cite [Lemmas 5.1, 5.2]{Agol}.

\medskip

Assume now that $M_i$ is a sequence of closed arithmetic hyperbolic 3-manifolds with $\inj(M_i) \geq \epsilon$ and $\rank(\pi_1(M_i))\le k$.  Passing to a subsequence and choosing suitable basepoints, we can assume by Proposition \ref{lem:degenerate} that $(M_i)$ converges in the Gromov-Hausdorff topology to a manifold $M_\infty$ with a degenerate end.  Theorem \ref{vigneras} implies that each $M_i$ covers an orbifold $\CO_i$ with $\lambda_1(\CO_i) \geq \frac 34$. Since $M_\infty$ has a degenerate end and $\lambda(\CO_i) \nrightarrow 0 $, Proposition \ref{orbifoldlaplace}  implies that the orbifolds $\CO_i$ cannot be pairwise distinct. In particular, the manifolds $(M_i)$ cannot be pairwise incommensurable. 

This proves the following theorem. 

\begin{sat}\label{arithmetic-finite}
For all $\epsilon$ and $k$ positive, there are only finitely many commensurability classes of closed arithmetic hyperbolic 3-manifolds $M$ with $\inj(M)\ge\epsilon$ and $\rank(\pi_1(M))\le k$.
\end{sat}

Using the same techniques more carefully, we can also show:

\begin{sat} For all $\epsilon $ and $k $ positive, there are only finitely many closed, arithmetic hyperbolic $3$-manifolds M with $\rank \pi_1(M)= k$ and $\inj(M) \geq \epsilon$ that do not fiber over either
\begin{enumerate}
\item ($k$ odd) $S^1$, with fibers of genus $\frac 12 (k-1)$, or
\item ($k$ even) $S^1 / (z \mapsto \bar z)$, with regular fibers of genus $k-2$.
\end{enumerate}
\label{fibers}
\end{sat}
\begin{proof}
Assume that there is an infinite sequence of counterexamples, i.e. a sequence $(M_i = \BH^ 3 / \Gamma_i)$ of closed arithmetic $3$-manifolds with $\inj (M_i) \geq \epsilon $ and $\rank (\pi_1 (M_i)) = k $ that do not fiber as above.  First, by the argument above we can pass to a subsequence in which every $M_i  $ covers some fixed orbifold $\CO = \BH^ 3 / \Gamma_\CO$.  Passing to another subsequence and using Proposition \ref{lem:degenerate}, we can assume that $(M_i) $ converges in the (based) Gromov-Hausdorff topology to a manifold $M_\infty = \BH^ 3 / \Gamma_\infty $ that is homeomorphic to $\Sigma \times \BR$ and has two degenerate ends. 

There is a small technical point we must address before proceeding.  Observe that after conjugation, we can arrange that each $\Gamma_i \subset \Gamma_\CO $.  With separate conjugations, we may assume that $\Gamma_i \to \Gamma_\infty $ in the Chabauty topology (Theorem \ref {GHvsCH}).  In fact, these properties can be arranged simultaneously: after conjugating $(\Gamma_i )$ into $\Gamma_\CO$, we must make sure to only conjugate $(\Gamma_i )$ by elements of $\Gamma_\CO $ when ensuring its convergence to a group $\Gamma_\infty $ with doubly degenerate quotient $M_\infty \cong \Sigma \times \BR $.  But since $ \Gamma_\CO \backslash \PSL_2 \BC $ is compact, any sequence $(\gamma_i) $ in $\PSL_2 \BC $ has the form $\gamma_i = o_i c_i $, where $o_i \in \Gamma_\CO $ and $(c_i) $ is pre-compact in $\PSL_2\BC $.  Passing to a subsequence where $c_i \to c_\infty \in \PSL_2\BC  $, we have that if $\gamma_i ^ { - 1 }\Gamma_i \gamma_i \to \Gamma_\infty$, then $o_i ^ { - 1 }\Gamma_i o_i \to c_\infty\Gamma_\infty c_\infty^ { - 1 }$.

As $\Gamma_i \subset \Gamma_\CO $ for $i = 1, 2,\ldots$, it follows that $\Gamma_\infty \subset \Gamma_\CO$ as well.  But $\Gamma_\CO $ is discrete, so the convergence $\Gamma_i \to \Gamma_\infty $ and the finite generation of $\Gamma_\infty $ imply $\Gamma_\infty \subset \Gamma_i $ for all large $i $.  This gives coverings $\phi_i : M_\infty \to M_i $; moreover, for any compact subset $K \subset M_\infty $ the convergence implies that $\phi_i | _K $ is injective for large $i $.  As $M_\infty \cong \Sigma \times \BR $ has two degenerate ends, it follows that all but finitely many $M_i $ fiber over $\BS^ 1 $ or $\BS^ 1 / (z \mapsto \bar z) $ with regular fibers homeomorphic to $\Sigma $.  This is stated in the appendix as Lemma \ref {embeddedcovering}, a sister to Thurston's covering theorem.

By Theorems 1.1 and 5.2 of \cite {Ian-rank}, the only way to have infinitely many $M_i $ with $\rank (\pi_1 M_i) = k $ that fiber with regular fibers $\Sigma $ is if 
\begin {enumerate}
\item $k $ is odd, $\Sigma $ has genus $\frac 12 (k - 1) $ and all but finitely many $M_i $ fiber over $\BS^ 1 $, or 
\item  $k $ is even, $\Sigma $ has genus $ k - 2 $ and all but finitely many $M_i $ fiber over $\BS^ 1 / (z \mapsto \bar z)$.
\end {enumerate}
In either case, the initial assumption that no $M_i $ fibers as above is clearly violated.  This completes the proof of Theorem \ref {fibers}, except for establishing Lemma \ref {embeddedcovering} in the appendix below.
\end{proof}

The following are immediate consequences of Theorem \ref{fibers}.

\begin{kor}
For $\epsilon, k >0$, there are only finitely many closed arithmetic hyperbolic 3-manifolds $M$ with $\inj(M)\ge\epsilon$ and $\rank(\pi_1(M))\le k$ that have the same $\BZ/2\BZ$-homology as $\BS^3$.
\end{kor}

\begin{kor}\label{arithmetic-finite-2}
There are only finitely many closed arithmetic hyperbolic 3-manifolds $M$ with $\inj(M)\ge\epsilon$ and $\rank(\pi_1(M))\le 3$.
\end{kor} 
This last result was first proven by Agol for $\rank \pi_1(M) =2$. 

\vspace{2mm}

The geometric version of Lehmer's conjecture \cite [Section 12.3] {Reid-book} states that the injectivity radius of a closed, arithmetic hyperbolic 3-manifold is bounded from below by some universal constant. A positive resolution of this conjecture would then remove all assumptions on injectivity radius from the theorems above.
\medskip

\section { Appendix }

In this  appendix, we prove a lemma used in the proof of Theorem \ref {fibers}.  It is a variant of Thurston's covering theorem (see \cite {Canary-covering}), a consequence of which is that if a closed hyperbolic $3$-manifold $N $ is covered by a doubly degenerate hyperbolic $3$-manifold $M $ homeomorphic to $\Sigma \times \BR $, then $N $ has a finite cover that fibers over the circle.  We show that if the covering is injective on a large subset of $M $, then $N $ itself fibers over the circle or over the orbifold $\BS^ 1 / (z \mapsto \bar z) $.
The proof will require some familiarity with simplicial hyperbolic surfaces.  We refer the reader to Canary's paper \cite {Canary-covering} for a good introduction to the subject, but it will be convenient to record some of their basic properties here.  

\begin{defi}\rm
Let $S$ be a surface with a metric $d $ that is hyperbolic except at a finite number of cone points, each with  angle at least $2\pi $. 
A \it simplicial hyperbolic surface \rm is a map $f : (S, d) \to M $ into a hyperbolic $3$-manifold that restricts to an isometry on the faces of some triangulation of $S $ whose vertices contain the cone points of $d $.  
\end{defi}

The reason for the assumption on cone angles is that then $(S, d) $ behaves like a negatively curved surface.  For instance, Bonahon \cite {Bonahon} noticed that these surfaces satisfy a bounded diameter lemma:

\begin{lem}\label {BDL}
If $(S, d) $ is a hyperbolic surface that has finitely many cone points with angles at least $2\pi $, then $\diam (S, d) \leq\frac{2|\chi (S)|}{ \inj (S, d)^2}.$
\end{lem}

A simplicial hyperbolic surface $f : S \to M $ is \it useful \rm if the associated triangulation on $S $ has only one vertex and one of its edges is sent to a closed geodesic in $M $.  The simple closed curve on $S $ formed by identifying the endpoints of that edge is said to be \it realized \rm by $f $.

\begin {lem} [Existence of useful surfaces \cite {Canary-covering}]\label{existence}
Let $M $ be a hyperbolic $3$-manifold without cusps. Then
 if $\gamma $ is a simple closed curve on $S $, any incompressible map $f : S \to M $ is homotopic to a useful simplicial hyperbolic surface realizing $\gamma $. 
 Also, if $E $ is a neighborhood of a degenerate end of $M $ then there is a sequence of useful simplicial hyperbolic surfaces $f_i : S \to E $ whose images leave every compact subset of $M $.

\end {lem}

Canary \cite {Canary-covering} showed that one can interpolate between useful simplicial hyperbolic surfaces. This is known as his \it filling theorem; \rm we state the following variant for incompressible maps.

\begin {sat}[Canary's filling theorem \cite {Canary-covering}]
Let $f, g : S \to M $ be two homotopic, incompressible, useful simplicial hyperbolic surfaces in a hyperbolic $3$-manifold $M $ without cusps.  Then there is a homotopy $$F : S \times [0, 1] \longrightarrow M, \ \ F (x, 0) = f (x) \text { and } F (x, 1) = g (x) $$ such that each $F (\cdot, t) $ is a simplicial hyperbolic surface.

\end{sat}

The past three results combine to give the following useful lemma.

\begin{lem}
Let $M $ be a doubly degenerate hyperbolic $3$-manifold homeomorphic to $\Sigma \times \BR $ with $\inj (M) \geq \epsilon $.  Then given $p \in M $, there is a simplicial hyperbolic surface passing through $p $.  Moreover, there is a \emph {useful} simplicial hyperbolic surface $f : S \to M $ with
$$\dist \big(p, f (S)\big) \leq \frac {\cosh^ { - 1 } \big(\frac { 4|\chi (S) | } {\epsilon^ 2 }\big) } \epsilon + \frac{2|\chi (S)|}{ \epsilon^2}. $$
\label {closeuseful}
\end{lem}
\begin {proof}
Lemma \ref{existence} allows us to find two useful simplicial hyperbolic surfaces deep in the two degenerate ends of $M $.  The homotopy between them given by the filling theorem must pass through $p $, so there is a simplicial hyperbolic surface $f: (S, d) \to M $  passing through $p $.  There is a simple closed curve on $S $ with length $L \leq 2\diam (S, d) $, which is bounded above by Lemma \ref{BDL}.  The closed geodesic freely homotopic to $f (\gamma) $ lies at most $\cosh^ { - 1 }(\frac L \epsilon )$ away from it, and there is a useful simplicial hyperbolic surface passing through this by Lemma \ref{existence}.
\end{proof}

We are now ready for the main result of this appendix.

\begin{lem} Given a closed, orientable surface $\Sigma $ and some $\epsilon > 0 $, there is a constant $D  =D (\Sigma,\epsilon) $ with the following property. Let $M $ be a hyperbolic $3$-manifold homeomorphic to $\Sigma \times \BR $ with two degenerate ends and $\inj (M) \geq \epsilon $.  If $\phi : M \to N $ is a Riemannian covering onto a closed $3$-manifold $N $ that restricts to an embedding on some ball of radius $D $ in $M $, then $N $ fibers over either $\BS^ 1 $ or the orbifold $\BS^ 1 / (z \mapsto \bar z) $ with regular fibers homeomorphic to $\Sigma $. \label {embeddedcovering}
\end{lem}
\begin {proof}
By Lemmas \ref {closeuseful} and \ref {BDL} and work of Freedman-Hass-Scott \cite [Theorem 2.5] {Canary-Minsky}, within unit distance of any point of $M $ there is a level surface with diameter at most $\frac { 2 |\chi (\Sigma) | } {\epsilon^ 2 } + 2 $. Choose $D= D (\Sigma,\epsilon)$ large enough so that there are level surfaces inside the ball on which $\phi $ is injective that are separated by a distance at least $$4 \Big [\frac {\cosh^ { - 1 } \big(\frac { 4|\chi (S) | } {\epsilon^ 2 }\big) } \epsilon + \frac{2|\chi (S)|}{ \epsilon^2}\Big] + 2. $$  Waldhausen's Cobordism Theorem \cite[Lemma 5.1]{Waldhausen} implies that these surfaces bound a submanifold $U \cong \Sigma \times [0, 1] $; this must also lie in the ball, so $\phi |_U $ is an embedding.

Lemma \ref {closeuseful} implies that there is a useful simplicial hyperbolic surface $f: (\Sigma, d) \to U $ that is a homotopy equivalence.  Lemma \ref {existence} then gives a sequence of homotopic useful simplicial hyperbolic surfaces $$  f_i : (\Sigma, d_i) \to M, \ \  i = 1, 2,\ldots $$ whose images leave every compact subset of $M $.  

 Since $\inj (M) \geq \epsilon $ and $(  f_i) $ are incompressible, $\inj (\Sigma, d_i) \geq \epsilon $ as well.  Proposition 2.1 from \cite {Ian-rank} and  Mahler's  compactness theorem then imply that there is a smooth hyperbolic metric $(\Sigma, d_{\text {hyp} }) $ and a sequence of homeomorphisms $r_i : (\Sigma,d_{\text {hyp} }) \to (\Sigma, d_i)$ that are uniformly lipschitz.  In particular, the sequence of compositions
$$\xymatrix{(\Sigma, d_{\text {hyp} }) \ar[r]^ -{ r_i } & \Sigma   \ar [r]^ {  f_i }  &  M \ar[r]^ {\phi } &N } $$
is uniformly lipschitz with images in a compact set, so Arzela-Ascoli's Theorem implies that it uniformly converges.   This shows that there are $\phi \circ f_j \circ r_j: \Sigma \to N $ and $\phi \circ f_k \circ r_k: \Sigma \to N $ that are homotopic by a homotopy with tracks of length less than $1 $.  Moreover, for later use let us arrange that $r_k \circ r_j^ { - 1 } : \Sigma \to \Sigma $ is not homotopic to the identity.  Fixing some curve $\gamma $ on $\Sigma $, the geodesic representative of $f_i \circ r_i (\gamma) $ lies a uniformly bounded distance from $f_i (\Sigma) $; therefore, if  the indices $j, k $ are far enough apart then  $f_j \circ r_j (\gamma) $ and $f_k \circ r_k (\gamma) $ cannot possibly be homotopic, implying that $r_k \circ r_j^ { - 1 } $ is not homotopic to the identity.

We now build a homotopy $F : \Sigma \times [0, 1] \to M $ between the maps $$\phi \circ f : \Sigma \longrightarrow M \  \text { and } \ \phi \circ f \circ r_k \circ r_j^ { - 1 } : \Sigma \longrightarrow M $$  through surfaces of diameter less than $\frac { 2 |\chi (S) | } {\epsilon^ 2 } + 2$.  It is constructed as a concatenation of three homotopies: 
\begin {itemize}
\item the composition with $\phi $ of a homotopy through simplicial hyperbolic surfaces from $f$ to $f_j $, as given by Canary's filling theorem,
\item a homotopy with tracks of length less than $1 $ from $\phi \circ f_j $ to $ \phi \circ f_k \circ r_k \circ r_j^ { - 1 } $, and
\item the composition with $\phi $ of a homotopy through simplicial hyperbolic surfaces from $f_k \circ r_k \circ r_j^ { - 1 } $ to $f \circ r_k \circ r_j^ { - 1 } $.
\end {itemize}
Because this is a homotopy through surfaces with diameter less than the distance between the components of $\partial \phi (U) $, the only way it can cross $\phi (U) $ is through surfaces contained in $\phi (U) $.  

Pick a homeomorphism $\phi (U) \cong \Sigma \times [0, 1 ] $ and form a quotient space $\bar M $ of $M $ by identifying two points of $\phi (U) $ if they have the same projection to $\Sigma $.  Then $\bar M \cong M $ and $\phi (U) $ projects to a surface $S\subset \bar M $ homeomorphic to $\Sigma $.  The homotopy $F $ descends to a homotopy $$\bar F : \Sigma \times [0, 1] \longrightarrow \bar M $$ such that no surface $\bar F (\Sigma,t) $ intersects both sides of a small regular neighborhood of $S $.  This implies that $\bar F $ is a concatenation of homotopies that lift to the manifold $\bar M \big | S $ obtained by cutting $\bar M $ along $S $.  

As $r_k \circ r_j ^ { -1 } $ is not homotopic to the identity, one of these lifts is a homotopy between components of $\partial (\bar M \big | S) $ that is not homotopic into $\partial (\bar M \big | S) $.  Waldhausen's Cobordism Theorem \cite{Waldhausen} implies that $\bar M \big | S $ is homeomorphic to $\Sigma \times [0, 1] $ or a trivial interval bundle over a non-orientable surface that is covered by $\Sigma $ with degree $2 $.  So, $\bar M \cong M $ fibers over $\BS^ 1 $ or $\BS^ 1 / (z \mapsto \bar z) $ with regular fibers homeomorphic to $\Sigma $.
\end{proof}

\bigskip

\noindent {\sc\small Ian Biringer, Yale University \\
\tiny Department of Mathematics\\
PO Box 208283\\
New Haven, CT 06520-8283.\\}
\medskip

\noindent {\sc \small Juan Souto, University of Michigan \\
\tiny Department of Mathematics\\
2074 East Hall   \\
 530 Church Street  \\
Ann Arbor, MI 48109-1043
\tiny }


\begin{thebibliography}{14}

\bibitem{Agol-tameness}
I. Agol, {\em Tameness of hyperbolic 3-manifolds}, 2004.

\bibitem{Agol}
I. Agol, {\em Finiteness of arithmetic Kleinian reflection groups},  International Congress of Mathematicians Vol. II, 2006.

\bibitem{Benedetti-Petronio}
R. Benedetti and C. Petronio, {\em Lectures on Hyperbolic Geometry}, Springer, 1992.

\bibitem{Ian-rank}
I. Biringer, {\em Geometry and rank of fibered hyperbolic 3-manifolds}, Algebr. Geom. Topol. 9, 2009.

\bibitem{Bonahon}
F. Bonahon,  \it Bouts des vari\'et\'es hyperboliques de dimension 3.  \rm Ann. of math 124, 1986.

\bibitem{BGLM}
M. Burger, T. Gelander, A. Lubotzky, and S. Mozes, {\em Counting hyperbolic manifolds}, Geom. Funct. Anal. 12, 2002.

\bibitem{Burger-Sarnak} M. Burger and P. Sarnak, {\em Ramanujan duals. II}, Invent. Math. 106, 1991.

\bibitem{Buser}
P. Buser, {\em A note on the isoperimetric constant}, Annales scientifiques de l'E.N.S. 15, 1982.

\bibitem{Calegari-Gabai}
D. Calegari and D. Gabai, {\em Shrinkwrapping and the taming of hyperbolic 3-manifolds}, J. Amer. Math. Soc. 19, 2006.

\bibitem{Canary-tameness}
  R. D. Canary, {\em Ends of hyperbolic 3-manifolds}, Journal Amer. Math. Soc. 6, 1993.
  
\bibitem{Canary-covering}
  R. D. Canary,
  {\em A covering theorem for hyperbolic 3-manifolds and its applications},
  Topology 35, 1996.

  
 \bibitem{Canary-Leininger}
R. D. Canary and C. Leininger, {\em  Kleinian groups with discrete length spectrum}, Bull. LMS 39, 2007

\bibitem {Canary-Minsky}
R. D. Canary and Y. Minsky, {\em On limits of tame hyperbolic $3$-manifolds}, J. Differential Geometry 43, 1996.

\bibitem{chavel}
I. Chavel, {\em Eigenvalues in Riemannian geometry}, Pure and applied mathematics, 1984.

\bibitem{Cheeger}
J. Cheeger, {\em A lower bound for the smallest eigenvalue of the Laplacian}, {Problems in Analysis}, Princeton University Press, 1970.

\bibitem{Cheegerfiniteness}
J. Cheeger, {\em Finiteness theorems for Riemannian manifolds},  Amer. J. Math. 92, 1970.

\bibitem{CLR}
T. Chinburg, D. Long and A. Reid, {\em Cusps of minimal non-compact arithmetic hyperbolic orbifolds}, to appear in Pure and Applied Math Quarterly.

\bibitem{GPW}
K. Grove, P. Petersen, J. Wu, {\em Geometric finiteness theorems via controlled topology},
Invent. Math. 99, 1990.

\bibitem{Hatchernotes}
A. Hatcher, {\em Notes on basic $3 $-manifold topology},
preprint.

\bibitem{long}  C. Maclachlan, D. Long and A. Reid,  \it Arithmetic fuchsian groups of genus
zero, \rm   Pure Appl. Math. Q. 2, 2006.

\bibitem{Reid-book}
C. Maclachlan and A. Reid, {\em The Arithmetic of Hyperbolic 3-Manifolds}, Springer, 2003.

\bibitem{Marden74}
A. Marden, {\em The geometry of finitely generated Kleinian groups}, Annals of Math. 99, 1974.

\bibitem {Japanese}
	K. Matsuzaki, M. Taniguchi, 
	{\em Hyperbolic Manifolds and Kleinian Groups}, 
	Oxford University Press, 1998.
\bibitem{Juan-rank}
J. Souto, {\em The Rank of the Fundamental Group of Hyperbolic 3-manifolds fibering over the circle}, in {\em The Zieschang Gedenkschrift}, Geometry and Topology Monographs, Vol. 14, 2008.

\bibitem {Thurstonbook}
W. Thurston, {\em Three-Dimensional Geometry and Topology}, Princeton University Press, 1997.

\bibitem {Thurstonnotes}
W. Thurston, {\em The Geometry and Topology of $3 $-Manifolds}, Princeton University lecture notes, 1980.

\bibitem{Vigneras}
M.-F. Vigneras, {\em Quelques remarques sur la conjecture $\lambda_1>\frac 14$}, Seminar on number theory, Paris 1981-82, Progr. Math., 38 Birkh\"auser, 1983.

\bibitem{Waldhausen}
F. Waldhausen, {\em On irreducible 3-manifolds which are sufficiently large}, Ann. of Math, 87, 1968. 

\bibitem{Wang}
H. C. Wang, {\em Topics on totally discontinuous groups}, Symmetric spaces, Dekker, 1972.

\bibitem{White}
M. White, {\em Injectivity radius and fundamental groups of hyperbolic 3-manifolds}, Comm. Anal. Geom. 10, 2002.

\end{thebibliography}
\end{document}